\documentclass[reqno]{amsart}
\usepackage[title]{appendix}

\usepackage[utf8]{inputenc}
\usepackage{amsmath,amsthm,amssymb,amsfonts}
\usepackage{enumitem}
\usepackage[pdfborder={0 0 0}]{hyperref}
\hypersetup{colorlinks,linkcolor={blue},citecolor={blue},urlcolor={red}} 
\usepackage{xcolor}
\usepackage{tikz}
\usetikzlibrary{arrows}
\usepackage{mathtools}
\usepackage{theoremref}
\usepackage{bm}
\usepackage{graphicx}
\usepackage{subfigure}
\usepackage{float}
\usepackage{bbm}
\usepackage{stmaryrd}
\usepackage{array}

\usepackage[paper=letterpaper,margin=.9in]{geometry}

\newtheorem{theorem}{Theorem}[section]
\newtheorem{lem}[theorem]{Lemma}
\newtheorem{prop}[theorem]{Proposition}
\newtheorem{cor}[theorem]{Corollary}

\newtheorem{rmk}[theorem]{Remark}

\numberwithin{equation}{section}

\newcommand{\RR}{\mathbb{R}}
\newcommand{\e}{\epsilon}

\newcommand{\totnum}{\mathfrak{m}}
\newcommand{\cls}{\mathfrak{b}}
\graphicspath{{./figs/}}
\newcommand{\cnf}{\mathfrak{n}}

\newcommand{\clsum}{N}

\newcommand{\cm}{\mathfrak{c}}
\newcommand{\zetab}{\boldsymbol{\zeta}}
\newcommand{\xib}{\boldsymbol{\xi}}
\newcommand{\fmt}{s_0}
\newcommand{\itmt}{t_0}

\newcommand{\htk}{\mathsf{q}}
\newcommand{\para}{\mathsf{p}}
\newcommand{\msE}{\mathsf{E}}

\newcommand{\spd}{v}
\newcommand{\mset}{\mathfrak{m}}
\newcommand{\cmspd}{\varphi}
\newcommand{\ft}{t}
\newcommand{\Zdc}{\mathcal{Z}}
\newcommand{\mra}{a^*}
\newcommand{\mrb}{b^*}
\begin{document}
	
	\title{Multi-point Lyapunov exponents of the Stochastic Heat Equation}
	\author{Yier Lin}
	\address[Yier Lin]{University of Chicago, Department of Statistics.} \email{ylin10@uchicago.edu}
	\begin{abstract}
We obtain the multi-point positive integer Lyapunov exponents of the Stochastic Heat Equation (SHE) and provide three expressions for them. We prove the result by matching the upper and lower bounds for the Lyapunov exponents. The upper bound is obtained by analyzing the contour integral formula in \cite{borodin2014macdonald}. For the lower bound, we apply an induction argument, relying on a tree recently appeared in \cite{tsai2023high}. The tree is related to the optimal trajectories of the Brownian motions in the Feynman-Kac formula. 
\end{abstract}
	\maketitle
	\section{Introduction}
In this paper, we consider the Stochastic Heat Equation (SHE) in one spatial dimension, 
	\begin{equation*}
	\partial_t Z (t, x) = \frac{1}{2} \partial_{xx} Z(t, x) + \xi(t, x) Z(t, x),  \qquad (t, x) \in \mathbb{R}_{> 0} \times \mathbb{R}, 
	\end{equation*}
where $\xi$ is the spacetime white noise. 

	The SHE is closely related to the Kardar-Parisi-Zhang (KPZ) equation \cite{kardar1986dynamic} 
	\begin{equation*}
	\partial_t h (t, x) = \frac{1}{2} \partial_{xx} h(t, x) + \frac{1}{2} (\partial_x h(t, x))^2 + \xi(t, x) 
	\end{equation*}
via the Hopf-Cole transform $Z(t, x) = \exp(h(t, x))$.
	The KPZ equation is a paradigm for modeling the random interface growth \cite{quastel2011introduction, corwin2012kardar}, a universal scaling limit of the weakly asymmetric interacting particle systems and a testing ground for the study of nonlinear stochastic PDEs.   

We focus on the SHE starting from the Dirac delta initial data $Z(0, \cdot) = \delta(\cdot)$. The SHE starting from the Dirac delta initial data has a unique mild solution that satisfies 
	\begin{equation*}
	Z(t, x) = \htk(t, x) + \int_0^t \int_{\mathbb{R}} \htk(t-s, x-y) Z(s, y) \xi(s, y) dsdy, \qquad (t, x) \in \mathbb{R}_{> 0} \times \mathbb{R}, 
	\end{equation*}
where $\htk(t, x) := \frac{1}{\sqrt{2\pi t}} e^{-\frac{x^2}{2t}}$ is the standard heat kernel. 
The stochastic integral against the spacetime white noise is interpreted in the It\^{o}'s sense. In addition, $Z$ is strictly positive on $(t, x) \in \mathbb{R}_{> 0} \times \mathbb{R}$ \cite{mueller1991support, 10.2307/42920501}. 
\subsection{Main result}
Throughout the paper, we consider the \emph{hyperbolic scaling} $Z_T (t, x) := Z(Tt, Tx)$. We compute the multi-point Lyapunov exponents of the SHE that are defined as the following limit   
\begin{equation*}
\lim_{T \to \infty} \frac{1}{T} \log \mathbb{E}\Big[\prod_{i = 1}^n Z_T (\ft, x_i)^{m_i}\Big].
\end{equation*}
	\begin{theorem}\label{thm:general}
		For fixed $n \in \mathbb{Z}_{\geq 1}$, $\ft > 0$, $\vec{x} = (x_1 < \ldots < x_n) \in \mathbb{R}^n$ and $\vec{m}  = (m_1, \ldots, m_n) \in \mathbb{Z}_{\geq 1}^n$, we have 
		\begin{equation*}
		\lim_{T \to \infty} \frac{1}{T} \log \mathbb{E}\Big[\prod_{i = 1}^n Z_T (\ft, x_i)^{m_i}\Big] = \gamma(\ft, \vec{x}, \vec{m}),
		\end{equation*}
where $\gamma := \gamma_1 = \gamma_2 = \gamma_3$, the expressions of the $\{\gamma_i\}_{i = 1}^3$ are given in \eqref{eq:gamma1} - \eqref{eq:gamma3}.
\end{theorem}
We set 
\begin{equation*}
\totnum := \sum_{i = 1}^n m_i \text{ and } u_k := x_j \text{ if } \sum_{i = 1}^{j-1} m_i < k \leq \sum_{i = 1}^{j} m_i.
\end{equation*}
We define
\begin{align}
\label{eq:gamma1}
\gamma_1 (\ft, \vec{x}, \vec{m}) &:= \inf\Big\{\sum_{i = 1}^\totnum \frac{\ft}{2} a_i^2 + \sum_{i = 1}^\totnum u_i a_{i}: a_{i} - a_{i+1} \geq 1, \,  i = 1, \ldots, \totnum - 1\Big\},\\
\label{eq:gamma2}
\gamma_2 (\ft, \vec{x}, \vec{m}) &:= \inf\Big\{\sum_{i = 1}^n \frac{m_i \ft}{2} (b_i +\frac{x_i}{\ft})^2 + \frac{(m_i^3 - m_i) \ft}{24} - \frac{m_i x_i^2}{2\ft}: b_i - b_{i+1} \geq \frac{m_i + m_{i+1}}{2}, i = 1, \ldots, n-1\Big\}.
\end{align}
The $a_i$ and $b_i$ in the infimums above are real numbers. We define $\mset_{\cls_j} := \sum_{i \in \cls_j} m_i$ and
\begin{equation}\label{eq:gamma3}
\gamma_3(t, \vec{x}, \vec{m}) := \sum_{j = 1}^{\cnf} \bigg(\frac{(\mset_{\cls_j}^3 - \mset_{\cls_j})\ft}{24} - \sum_{k, \ell \in \cls_j, k < \ell} \frac{m_k m_\ell}{2} |x_k - x_\ell| - \frac{(\sum_{k \in \cls_j} m_k x_k)^2}{2  \ft \mset_{\cls_j} }\bigg). 
\end{equation}
The positive integer $\cnf$ and sets $\{\cls_j\}_{j = 1}^{\cnf}$ (which form a partition of $\{1, \ldots, n\}$) 
will be defined in Section \ref{sec:inertia}.

When $n$ is small, we have more explicit expressions for the Lyapunov exponents. 
\begin{cor}\label{cor:explicit}
Take $n = 1$ in Theorem \ref{thm:general}, we obtain the one-point Lyapunov exponents of the SHE
\begin{equation*}
\lim_{T \to \infty} \frac{1}{T} \log \mathbb{E}[Z_T(t, x_1)^{m_1}] = \frac{(m_1^3-m_1)t}{24} - \frac{m_1 x_1^2}{2t}.
\end{equation*}
Take $n = 2$ in Theorem \ref{thm:general}, we have
\begin{equation*}
\lim_{T \to \infty} \frac{1}{T} \log \mathbb{E}\Big[\prod_{i = 1}^2 Z_T (\ft, x_i)^{m_i}\Big] = 
\begin{cases}
\frac{((m_1 + m_2)^3 - (m_1 + m_2))\ft}{24} - \frac{m_1 m_2 (x_2 - x_1)}{2}  - \frac{(m_1 x_1 + m_2 x_2)^2}{2(m_1 + m_2) \ft} 
&  \text{if } 0 < \frac{x_2-x_1}{t} \leq \frac{m_1 + m_2}{2}, \\
\frac{(m_1^3 - m_1)\ft}{24} + \frac{(m_2^3 - m_2)\ft}{24} - \frac{m_1 x_1^2}{2 \ft}  - \frac{m_2 x_2^2}{2 \ft}, &  \text{if } \frac{x_2-x_1}{t} > \frac{m_1 + m_2}{2}.
\end{cases}
\end{equation*}
\end{cor}
As a byproduct, Theorem \ref{thm:general} shows that $\gamma_1 = \gamma_2 = \gamma_3$, which is surprising since the expressions of them are quite different. In the following, we characterize the minimizer of the variational expression $\gamma_1$. We define a map $f: \{1, \ldots, \totnum\} \to \{1, \ldots, n\}$ such that $u_i = x_{f(i)}$. 
\begin{cor}\label{cor:minimizer}
	The infimum in $\gamma_1$ has a unique minimizer $(\mra_1, \ldots, \mra_{\totnum})$. In addition, we have $\mra_{i} - \mra_{i+1} = 1$ if there exists $j \in \{1, \ldots, \cnf\}$ such that $f(i), f(i+1) \in \cls_j$. We have $\mra_{i} - \mra_{i+1} > 1$ if there does not exist such $j$. 
\end{cor} 
\begin{rmk}
The moments of the SHE admit contour integral formulas \eqref{eq:contourint}, which follow from \cite[Proposition 6.2.3]{borodin2014macdonald}. To analyze the contour integral, one can deform the contours so that they pass the critical points of the integrand. Meanwhile, one might cross the poles generated from the denominator in \eqref{eq:contourint} and pick up residues. By book-keeping the residues that we pick up, we have a residue expansion of the contour integral. The minimizer of $\gamma_1$ indicates the dominating term in the residue expansion as $T \to \infty$. When $n = 1$, $f(i) = f(i+1)$ for every $i$. Corollary \ref{cor:minimizer} implies that $\mra_i - \mra_{i+1} = 1$ for all $i \in \{1, \ldots, \totnum - 1\}$. This shows that one can obtain the dominating term in the residue expansion of \eqref{eq:contourint} by picking up all the residues when we cross the poles, which coincides with the observation in \cite[Lemma 4.1]{corwin2020kpz}. However, the same phenomenon might not hold when $n > 1$. By Corollary \ref{cor:minimizer}, it is possible that $\mra_{i} - \mra_{i+1} > 1$ for some $i \in \{1, \ldots, \totnum - 1\}$.   	
\end{rmk}
	\subsection{Proof idea}\label{sec:method}
Let us explain the idea for proving Theorem \ref{thm:general}. 
We respectively show the upper bound 
\begin{equation}\label{eq:upbd}
\limsup_{T \to \infty} \frac{1}{T} \log\mathbb{E}\Big[\prod_{i = 1}^n Z_T (\ft,  x_i)^{m_i}\Big] 
\leq \gamma_{1} (\ft, \vec{x}, \vec{m}), 
\end{equation}
and the lower bound
\begin{equation}\label{eq:lwbd}
\liminf_{T \to \infty} \frac{1}{T} \log \mathbb{E}\Big[\prod_{i = 1}^n Z_T (\ft, x_i)^{m_i}\Big] \geq \gamma_{3} (\ft, \vec{x}, \vec{m}). 
\end{equation}
After obtaining these bounds, we conclude Theorem \ref{thm:general} by showing $\gamma_1 \leq \gamma_2 \leq \gamma_3$. In the following, we focus on explaining the idea for obtaining the bounds \eqref{eq:upbd} and \eqref{eq:lwbd}. 
\subsubsection{Upper bound}
It is known that the moments of the SHE $\mathbb{E}[\prod_{i = 1}^n Z (t, x_i)]$ solve a PDE called the \emph{delta Bose gas} \cite[Section 6.2]{borodin2014macdonald}. More precisely, let $U(t; x_1, \ldots, x_n) := \mathbb{E}[\prod_{i = 1}^n Z(t, x_i)]$, we have
\begin{equation*}
\partial_t U = \frac{1}{2} \Delta U + \frac{1}{2} \sum_{i \neq j} \delta(x_i - x_j) U. 
\end{equation*}
The solution to the above PDE (starting from certain initial data) admits a contour integral expression \cite[Proposition 6.2.3]{borodin2014macdonald}. By a straightforward analysis of the contour integral, we obtain the upper bound.  
\subsubsection{Lower bound}
Apply the Feynman-Kac formula to solve the delta Bose gas, we can express the moments of the SHE as expectations of the Brownian local times, namely, $\mathbb{E}[\prod_{i = 1}^n Z_T (\ft, x_i)^{m_i}]$ is equal to
\begin{equation}\label{eq:feynman-kac}
\mathbb{E}\Big[\exp\Big(\sum_{1 \leq i <  j \leq \totnum} \int_0^{T\ft} \delta(W^i_s - W^j_s) ds\Big) \prod_{i = 1}^\totnum \delta(W^i_{Tt})\Big],
\end{equation}
the $\{W^i\}_{i=  1}^\totnum$ are independent Brownian motions. Note that $W^j_0 = T u_j$, namely $W^j_0 = Tx_i$ if $\sum_{k = 1}^{i-1} m_k < j \leq \sum_{k = 1}^i m_k$. Hence, there are $m_i$ Brownian motions starting from the location $T x_i$. We want to understand the optimal trajectories of the Brownian motions, i.e. the deterministic trajectories where the Brownian motions stay around contribute most to the expectation \eqref{eq:feynman-kac}.

The Dirac function at the terminal time in \eqref{eq:feynman-kac} forces the Brownian motions to end up at $0$. In order to contribute more to the Brownian local times on the exponential, the Brownian motions tend to move close to one another before the terminal time. Once they are close, they no longer want to be apart. On the other hand, they avoid traveling too fast to become close due to their transition probabilities. Consider the macroscopic picture by scaling both space and time by $T$, the above discussion suggests that the optimal trajectories of the Brownian motions form a tree as shown in Figure \ref{fig:tree}. In Section \ref{sec:inertia}, we will characterize precisely this tree using the attractive Brownian particles. 
Following \cite[Section 2.3]{tsai2023high}, we call this tree the \emph{optimal clusters}. 
\begin{figure}[ht]
	\centering
	\includegraphics[scale = 12]{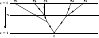}
	\caption{The optimal trajectories form a tree.}
	\label{fig:tree}
\end{figure}

We proceed to explain how to obtain the lower bound \eqref{eq:lwbd}. The idea is to apply an induction argument over $n$, which is the number of different locations that the Brownian motions start from. 
Let $s_0$ be the first time that the number of different points in the optimal clusters is smaller than $n$, i.e. the first time that two branches in the tree merge. 
We break the integral in \eqref{eq:feynman-kac} into integrals over the time interval $[0, s_0]$ and $[s_0, t]$ (remember the time and space have been scaled by $T$) and restrict on the event that the Brownian motions at time $s_0$ stay around the optimal clusters. Since the number of different points in the optimal clusters at time $s_0$ is less than $n$, we can lower bound the contribution of the Brownian motions in time $[s_0, t]$ using the induction hypothesis. In addition, we can  lower bound the contribution in time $[0, s_0]$ by dropping the Brownian local times if the two Brownian motions start from different $x_i$, this is fine since the optimal trajectories of the Brownian motions with different starting locations do not overlap before time $s_0$. Put together the lower bounds for the integrals over $[0, s_0]$ and $[s_0, t]$, we obtain the desired lower bound \eqref{eq:lwbd}.

We remark that the actual proof in Section \ref{sec:pflwbd} is slightly different from what has been explained above, although the idea behind is quite similar. Instead of using \eqref{eq:feynman-kac}, we rely on the semigroup identity of the SHE to carry out the proof. 
%
\subsection{Discussion}
Let us discuss three related work. In \cite{corwin2020kpz}, the authors obtained the one-point Lyapunov exponents of the SHE under Dirac delta initial data by applying a residue expansion to the contour integral \eqref{eq:contourint}. Since there are many poles in the contour integral, book-keeping the residues when deforming the contour is not easy. The advantage in the one-point situation is that since $u_1 = \ldots = u_{\totnum}$, the exponential function in \eqref{eq:contourint} is symmetric in $z_1, \ldots, z_{\totnum}$, and we have a simple residue expansion thanks to \cite[Proposition 6.2.7]{borodin2014macdonald}. For the multi-point $(n > 1)$ situation, the residue expansion seems much harder to be analyzed due to the lack of symmetry. Instead of relying on the contour integral formula to derive the full asymptotic of the Lyapunov exponents, we only use it to obtain an (optimal) upper bound.  
 
\cite{ganguly2022sharp} obtained the sharp bounds for the two-point upper tail probability of the KPZ equation in finite large time using the Gibbsian line ensembles \cite{corwin2014brownian, corwin2016kpz}. The method therein should work for obtaining the multi-point upper tail bounds and could be applied to obtain the multi-point (even for non-integer valued) Lyapunov exponents of the SHE. This approach, however, seems a bit indirect and we are not sure whether it will lead to a simple expression of the Lyapunov exponents.

The work \cite{tsai2023high} studied the Lyapunov exponents of the SHE 
$\frac{1}{N^3 T} \log \mathbb{E}[\prod_{i = 1}^n Z_T(\ft, Nx_i)^{Nm_i}]$
under the \emph{high-moment regime} 
$N^2 T \to \infty \text{ and } N \to \infty$
and
obtained the multi-point Lyapunov exponents by studying the large deviations of the attractive Brownian particles. 
In this paper, we consider the hyperbolic scaling $N = 1$ and $T \to \infty$, 
which is different from the high-moment regime. 
Our method relies on the exact formula and an induction argument, which is new. We have not seen the multi-point Lyapunov exponents of the SHE being studied in the context of hyperbolic scaling before.
\subsection{Literature Review}
The one-point Lyapunov exponents of the SHE with different initial conditions have been studied in \cite{bertini1995stochastic, chen2015precise, corwin2020kpz, das2021fractional, ghosal2023lyapunov}. The moments and Lyapunov exponents 
are useful for studying the property of \emph{intermittency} \cite{gartner1990parabolic, gartner2007geometric, khoshnevisan2014analysis, chen2015precise}, the density function and large deviations of the SHE (and its variant), see \cite{conus2013chaotic, georgiou2013large, borodin2014moments,  chen2015moments, hu2015stochastic, janjigian2015large, khoshnevisan2017intermittency, corwin2020kpz, chen2021regularity, das2021fractional, lin2021lyapunov, das2022upper, lamarre2022moment, hu2022asymptotics, ghosal2023lyapunov}. 

The Lyapunov exponents of the SHE are closely related to the upper tail bounds and Large Deviation Principle (LDP) of the KPZ equation. 
The one-point tail bounds and LDP of the KPZ equation with different boundary and initial conditions have been intensively studied recently in the physics work \cite{le2016exact, le2016large, krajenbrink2017exact, sasorov2017large, corwin2018coulomb, krajenbrink2018large, krajenbrink2018simple, krajenbrink2018systematic, krajenbrink2019beyond, krajenbrink2019linear, le2020large} and  mathematics work \cite{ corwin2020lower, corwin2020kpz, cafasso2021airy, das2021fractional, kim2021lower, lin2021lyapunov, cafasso2022riemann, tsai2022exact, das2023law, ghosal2023lyapunov}. The two (and potentially multi)-point upper tail bounds and the terminal-time limit shapes of the KPZ equation have been studied in \cite{ganguly2022sharp}.  
The Freidlin-Wentzell LDP/weak noise theory has been used to study the one-point LDP and most probable shapes of the KPZ equation, see the physics work \cite{kolokolov2007optimal, kolokolov2008universal, kolokolov2009explicit, meerson2016large, kamenev2016short, meerson2017height, meerson2018large, smith2018exact, smith2018finite, asida2019large, smith2019time} and mathematics work \cite{lin2021short, lin2021short, lin2022lower, gaudreau2023kpz}. The connection between the Freidlin-Wentzell LDP/weak noise theory and the integrable PDE has been studied in the physics work \cite{krajenbrink2020painleve, krajenbrink2021inverse, krajenbrink2022inverse, krajenbrink2023crossover} and mathematics work \cite{tsai2022integrability}. 
The LDP and spacetime limit shapes of the KPZ equation in the deep upper tails have been studied recently in \cite{tsai2023high, lin2023spacetime}.  

	\subsection*{Outline}
	In Section \ref{sec:upbd}, we prove the upper bound \eqref{eq:upbd}. In Section \ref{sec:lwbd}, we prove the lower bound \eqref{eq:lwbd}. 
	In Section \ref{sec:continuity}, we establish a continuity result, which is a technical input for proving the lower bound. In Section \ref{sec:ordering}, we prove Theorem \ref{thm:general} and its corollaries.
	\subsection*{Acknowledgment}
We thank Ivan Corwin, Greg Lawler, Russel Lyons and Li-Cheng Tsai for the helpful discussion. We thank Guillaume Barraquand for telling us a rigorous approach for proving the exact formula \eqref{eq:contourint}. 
We acknowledge the support of a research fund from Department of Statistics of the University of Chicago.
	\section{The upper bound}
	\label{sec:upbd}
In this section, we will prove the upper bound \eqref{eq:upbd}, which is stated as the following proposition. The major input is a contour integral expression for the moments of the SHE. 
\begin{prop}\label{prop:upbd}
Under the same setting as Theorem \ref{thm:general}, we have		
\begin{equation*}
\limsup_{T \to \infty} \frac{1}{T} \log \mathbb{E}\Big[\prod_{i = 1}^n Z_T(\ft, x_i)^{m_i}\Big] 
\leq \gamma_{1} (t, \vec{x}, \vec{m}). 
\end{equation*}
\end{prop}
\begin{proof}
Recall that $u_k := x_j$ if $\sum_{i = 1}^{j-1} m_i < k \leq \sum_{i = 1}^{j} m_i$. It is straightforward that \begin{equation*}
\mathbb{E}\Big[\prod_{i = 1}^n Z_T(\ft, x_i)^{m_i}\Big] = \mathbb{E}\Big[\prod_{i = 1}^\totnum Z_T (\ft, u_i)\Big].
\end{equation*} 
By \cite[Proposition 6.2.3]{borodin2014macdonald}, we know that \begin{equation}\label{eq:contourint}
\mathbb{E}\Big[\prod_{i = 1}^\totnum Z_T (\ft, u_i)\Big] = \frac{1}{(2\pi \mathbf{i})^{\totnum}} \int \ldots \int \prod_{1 \leq A < B \leq \totnum} \frac{z_A  -z_B}{z_A - z_B - 1} \exp\bigg({\sum_{j = 1}^{\totnum} T \ft z_j^2/2 + \sum_{j = 1}^{\totnum} Tu_j z_j}\bigg) \prod_{j = 1}^\totnum dz_j,
\end{equation}
the $z_j$-contour is given by $a_j + \mathbf{i} \mathbb{R}$ and $\{a_k\}_{k = 1}^{\totnum}$ can be any real numbers satisfying $a_{k} - a_{k+1} > 1$, $k = 1, \ldots, \totnum-1$. We remark that the proof for the identity above in \cite[Proposition 6.2.3.]{borodin2014macdonald} is not fully rigorous. However, one can combine \cite[Proposition 5.4.8]{borodin2014macdonald} together with \cite[Corollary 1.7]{nica2021intermediate} to obtain a rigorous proof. 

Apply a change of variable $z_j = a_j + \mathbf{i} y_j$ for $j = 1, \ldots, \totnum$, the triangle inequality $\int \ldots \leq \int |\ldots|$, and 
	$|\frac{z_A - z_B}{z_A - z_B - 1}| \leq \frac{a_{A} - a_{B}}{a_{A} - a_{B} - 1}$, we have 
	\begin{align*}
	\mathbb{E}\Big[\prod_{i = 1}^\totnum Z_T (\ft, u_i)\Big] &\leq \frac{1}{(2\pi)^{\totnum}} \prod_{1 \leq A < B \leq \totnum}  \Big|\frac{a_A - a_B}{a_A - a_B - 1}\Big| \exp\bigg({\sum_{j = 1}^{\totnum} T \ft a_j^2/2 + \sum_{j = 1}^{\totnum} T u_j a_j}\bigg) \int_{\mathbb{R}^{\totnum}} \prod_{j = 1}^\totnum e^{-T\ft y_j^2/2} dy_j
	\\
	&= (2\pi T \ft)^{-\totnum/2} \prod_{1 \leq A < B \leq \totnum} \Big|\frac{a_A - a_B}{a_A - a_B - 1}\Big| \exp\bigg(\sum_{j = 1}^{\totnum} T \ft a_j^2/2 + \sum_{j = 1}^{\totnum} T u_j a_j\bigg). 
	\end{align*}
	Take the logarithm of both sides above, divide them by $T$ and let $T \to \infty$, we know that 
	\begin{equation*}
	\limsup_{T \to \infty} \frac{1}{T} \log \mathbb{E}\Big[\prod_{i = 1}^\totnum Z_T (\ft, u_i)\Big] \leq \sum_{j = 1}^\totnum \frac{\ft a_j^2}{2} + \sum_{j = 1}^\totnum u_j a_j. 
	\end{equation*}
The inequality above holds for $a_1, \ldots, a_\totnum$ satisfying $a_{k} - a_{k+1} > 1$, $k = 1, \ldots, \totnum-1$. By the continuity of the right hand side above in $a_1, \ldots, a_\totnum$, the inequality also holds for any $a_1, \ldots, a_\totnum$ satisfying $a_{k} - a_{k+1} \geq 1$, $k = 1, \ldots, \totnum-1$. Take the infimum of the right hand side over $a_1, \ldots, a_{\totnum}$ under this constraint, we conclude the proposition. 
\end{proof}

\section{The lower bound}\label{sec:lwbd}

	
In this section, we will prove the lower bound \eqref{eq:lwbd}, which is stated as the following proposition. 
\begin{prop}\label{prop:lwbd}
Under the same setting as Theorem \ref{thm:general}, we have	
\begin{equation*}
\liminf_{T \to \infty} \frac{1}{T} \log \mathbb{E}\Big[\prod_{i = 1}^n Z_T(t, x_i)^{m_i}\Big] \geq \gamma_{3} (t, \vec{x}, \vec{m}). 
\end{equation*}
\end{prop}
\subsection{The inertia clusters and optimal clusters}\label{sec:inertia}
We recall the \emph{inertia clusters} $\zetab = \{\zetab_i\}_{i = 1}^n$ and \emph{optimal clusters} $\xib = \{\xib_i\}_{i = 1}^n$ from \cite[Section 2.3]{tsai2023high}, see Figure \ref{fig:cls} for a visualization. Let us first define the inertia clusters, then we use them to define the optimal clusters. 

The inertia clusters are the trajectories of the point masses which run backward in time (compared with the time unit used in $Z_T$) from $s = 0$ to $s = t$. We say the inertia clusters start from $(\vec{x}, \vec{m})$ if at time $s = 0$, they start from the point masses with weight $m_i$ and location $x_i$, $i = 1, \ldots, n$. 
As time evolves, the point mass with weight $m_i$ will travel with a constant speed $\phi_i := \frac{1}{2}(m_{i+1} + \ldots + m_n) - \frac{1}{2} (m_{1} + \ldots + m_{i-1})$. When the point masses with weights $m_i$ and $m_{i+1}$ and speeds $\phi_i$ and $\phi_{i+1}$ collide, they merge into a single point mass with weight $m_i + m_{i+1}$ and travel with speed $\frac{m_i \phi_i + m_{i+1} \phi_{i+1}}{m_i + m_{i+1}}$, this follows the conservation of momentum. 
Let $\zetab_i: [0, t] \to \mathbb{R}, i = 1, \ldots, n$ denote the trajectories of the inertia clusters. Examine the point masses which have merged between time $[0, t]$, we obtain a partition of $\{1, \ldots, n\}$. Denote the partition to be $\mathfrak{B} = \{\cls_i\}_{i = 1}^{\cnf}$, where 
\begin{equation*}
\cnf := \text{the number of different clusters at time $t$}.
\end{equation*} 
We can order $\cls_1, \ldots, \cls_\cnf$ such that for $i < j$, the elements of $\cls_i$ are smaller than those in $\cls_j$. Note that we have $\zetab_i (t) = \zetab_j (t)$ if and only if $i, j$ belong to the same $\cls_k$ for some $k$. We define $\zetab_{\cls_k} (t) := \zetab_i (t) = \zetab_j (t)$.   

We proceed to define the optimal clusters $\xib$. Let $\spd_j  := \zetab_{\cls_j} (t)/t$ for $j = 1, \dots, \cnf$. We define the trajectories of the optimal clusters $\{\xib_i\}_{i = 1}^n$ by applying a constant drift to the inertia clusters: $\xib_i (s) := \zetab_i (s) - \spd_j s$ for $i \in \cls_j$. Note that we have $\xib_i (t) = 0$ for $i = 1, \ldots, n$. 

Let us discuss non-rigorously why the optimal clusters should be the optimal trajectories of the Brownian motions for the expectation in \eqref{eq:feynman-kac}. Following \cite[Appendix A]{tsai2023high}, we use the Tanaka's formula to write $\int_{0}^{Tt} \delta (W_s^i - W_s^j) ds =  -\frac{1}{2}\int_0^{Tt} \text{sgn}(W_{s}^i - W_{s}^j) d(W_s^i - W_s^j) + \frac{1}{2}|W_s^i - W_s^j|\big|_{s = 0}^{s = Tt}$ with $\text{sgn}(x) = \mathbf{1}_{\{x > 0\}} - \mathbf{1}_{\{x < 0\}}$. The quadratic variation of the stochastic integral on the resulting exponential is equal to $\frac{(\totnum^3 - \totnum) t}{12}$. Apply Girsanov theorem, \eqref{eq:feynman-kac} equals
\begin{equation}
\label{eq:feynman-kac1}
\exp\Big(\frac{(\totnum^3 - \totnum) T t}{24} - \sum_{1 \leq k < \ell \leq n} \frac{T m_k m_\ell}{2} |x_k - x_\ell|\Big)	\mathbb{E}\Big[\exp\Big(\frac{1}{2}\sum_{1 \leq i <  j \leq \totnum} |X^i_{Tt} - X^j_{Tt}|\Big) \prod_{i = 1}^\totnum \delta(X^i_{Tt})\Big],
\end{equation}
where $dX_s^i = \frac{1}{2} \sum_{j = 1}^n \text{sgn}(X_s^j - X_s^i) + dW_s^i$ with $X_0^i = T u_i$. Note that the time is of order $T$ and the diffusion has an order of $\sqrt{T}$, which is negligible compared with the drift as $T \to \infty$. Drop the diffusion, we refer to the resulting $\{X^i\}_{i = 1}^{\totnum}$ as the \emph{attractive Brownian particles}. The $i$-th Brownian particle has the drift $\frac{1}{2} \sum_{j = 1}^n \text{sgn}(X_s^j - X_s^i)$. In addition, when two Brownian particles meet, they stay together afterwards. Hence, the weight of the point masses in the inertia clusters can be viewed as the number of Brownian particles staying together. Scale the space and time by $T$, the trajectories of $\{X^i\}_{i = 1}^\totnum$ are given by the inertia clusters $\zetab$. The Dirac delta function in \eqref{eq:feynman-kac1} forces the clusters to end at $0$. The most economic way to fulfill this (in terms of transition probability) is to apply a constant drift to each group in the inertia clusters. This leads to the optimal clusters $\xib$.

For our proof in the paper, we only need the definition of $\xib$, the discussion in the previous paragraph will not be used.  
\begin{figure}[ht]
\centering
\includegraphics[scale = 12]{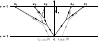}
\caption{An illustration of the inertia clusters $\zetab$ (gray) and the optimal clusters $\xib$ (black) when $n = 5$.  In the figure, we have $\cnf = 2$. The partition of $\{1, 2, 3, 4, 5\}$ is given by $\mathfrak{B} = \{\cls_1, \cls_2\}$ where $\cls_1 = \{1, 2, 3\}$ and $\cls_2 = \{4, 5\}$.} 
\label{fig:cls}
\end{figure}
\subsection{Proof of Proposition \ref{prop:lwbd}}
\label{sec:pflwbd}
We use the Feynman-Kac formula to introduce a four parameter version of $Z_T$. For $r < t$, we set 
\begin{equation}\label{eq:fmpolymer}
Z_T (r, y; t, x) := \mathbb{E}\Big[\exp\Big(\int_{0}^{T(t - r)} \xi(Tt-s, W_s) ds\Big) \delta(W_{T(t-r)} - Ty)\Big].
\end{equation}
The exponential above is the Wick exponential and $W$ is a standard Brownian motion starting from $W_0 = Tx$.
In particular, we have $Z_T (t, x) = Z_T(0, 0; t, x)$. 
By \eqref{eq:fmpolymer}, we have the semigroup identity 
\begin{equation*}
Z_T (t, x) =  \int_\mathbb{R} T Z_T (r, y) Z_T (r, y; t, x) dy.
\end{equation*}  
 In the following, 
\begin{equation*}
A \gtrsim B \text{ means } \liminf_{T \to \infty} \frac{1}{T} \log(A/B) \geq 0.
\end{equation*}
Under this notation, showing Proposition \ref{prop:lwbd} is the same as showing  \begin{equation}\label{eq:lwbdequi}
\mathbb{E}\Big[\prod_{i = 1}^n Z_T(t, x_i)^{m_i}\Big] \gtrsim e^{T\gamma_{3} (t, \vec{x}, \vec{m})}.
\end{equation}
\begin{proof}[Proof of Proposition \ref{prop:lwbd}]
We apply an induction over $n$ for proving \eqref{eq:lwbdequi}. When $n = 1$, the result follows from \cite[Lemma 4.1]{corwin2020kpz} 
and the stationarity  of $\{\frac{Z_T(t, x)}{\htk(Tt, Tx)}\}_{x \in \mathbb{R}}$. When $n > 1$, we only need to prove \eqref{eq:lwbdequi} under the \emph{induction hypothesis} that  
$\mathbb{E}[\prod_{i = 1}^{n'} Z_T(t', x'_i)^{m'_i}] \gtrsim 
e^{T\gamma_{3} (t, \vec{x}', \vec{m}')}$
holds for any $t' > 0$, $\vec{x}' = (x'_1 < \cdots < x'_{n'}) \in \mathbb{R}^{n'}$, $\vec{m}' = (m'_1, \ldots, m'_{n'}) \in \mathbb{Z}_{\geq 1}^{n'}$ with $n' < n$.

Recall the number $\cnf$ and the partition $\{\cls_j\}_{j = 1}^{\cnf}$ from Section \ref{sec:inertia}. We divide the proof into two cases: $\cnf > 1$ and $\cnf  = 1$. When $\cnf > 1$, we have $|\cls_j| < n$ for each $j \in \{1, \dots, \cnf\}$. 
By the induction hypothesis, we know that 
\begin{equation}
\label{eq:clslwbd}
\mathbb{E}\Big[\prod_{k \in \cls_j} Z_T (t, x_k)^{m_k}\Big] 
\gtrsim 
e^{T\gamma_3 (t, \vec{x}_{\cls_j}, \vec{m}_{\cls_j})},
\end{equation}
where $\vec{x}_{\cls_j}, \vec{m}_{\cls_j}$ are the vectors for $\{x_i\}_{i \in \cls_j}, \{m_i\}_{i \in \cls_j}$. 
Note that the trajectories of the inertia  clusters starting from $(\vec{x}_{\cls_j}, \vec{m}_{\cls_j})$ are still given by $\{\zetab_i\}_{i \in \cls_j}$. Hence, if we start the Brownian particles from $(\vec{x}_{\cls_j}, \vec{m}_{\cls_j})$,  there is only one cluster at time $s = t$. We have 
\begin{equation*}
\gamma_3 (t, \vec{x}_{\cls_j}, \vec{m}_{\cls_j}) = \frac{(\mset_{\cls_j}^3 - \mset_{\cls_j})t}{24} - \sum_{k, \ell \in \cls_j, k < \ell} \frac{m_k m_\ell}{2} |x_k - x_\ell| - \frac{(\sum_{k \in \cls_j} m_k x_k)^2}{2  t \mset_{\cls_j}}.
\end{equation*}  

Apply Lemma \ref{lem:FKGprod} and then  \eqref{eq:clslwbd} for $j = 1, \ldots, \cnf$, we obtain the desired inequality
\begin{equation*}
\mathbb{E}\Big[\prod_{i = 1}^n Z_T(t, x_i)^{m_i}\Big] \geq 
\prod_{j = 1}^{\cnf} \mathbb{E}\Big[ \prod_{k \in \cls_j} Z_T (t, x_k)^{m_k}\Big] \gtrsim \exp\Big(\sum_{j = 1}^\cnf T\gamma_3 (t, \vec{x}_{\cls_j}, \vec{m}_{\cls_j})\Big) = e^{T\gamma_3 (t, \vec{x}, \vec{m})}.
\end{equation*}

We proceed to prove Proposition \ref{prop:lwbd} when $\cnf = 1$. Define $\fmt = \inf\{s: |\{\xib_1 (s), \dots, \xib_n (s)\}| < n\}$, which is the first time that a merging happens among the point masses in the optimal clusters. Fix $\delta > 0$, apply the semigroup identity
$Z_T(t, x_{k}) =  \int_{\mathbb{R}} T Z_T(t-\fmt, y_{k, \ell}) Z_T (t-\fmt, y_{k, \ell}; t, x_k) dy_{k, \ell}$ for $\ell = 1, \ldots, m_k$, use Fubini's theorem to exchange the expectation and integral, then use the independence between $Z_T (t-s_0, \cdot)$ and $Z_T(t-s_0, \cdot; t, \cdot)$, and finally restrict the domain of integral for $y_{k, \ell}$ to $[\xib_k (s_0) - \delta, \xib_k (s_0) + \delta]$, we know that $\mathbb{E}[\prod_{k = 1}^n Z_T (t, x_k)^{m_k}]$ is lower bounded by
\begin{equation}\label{eq:twoexp}
\int_{\RR^\totnum} T^{\totnum} \mathbb{E}\Big[\prod_{k = 1}^n \prod_{\ell = 1}^{m_k} Z_T (t-\fmt, y_{k, \ell})\Big] \mathbb{E}\Big[\prod_{k = 1}^n \prod_{\ell = 1}^{m_k} Z_T (t-\fmt, y_{k, \ell}; t, x_k)\Big] \prod_{k = 1}^n \prod_{\ell = 1}^{m_k} \mathbf{1}_{\{|y_{k, \ell} - \xib_k (s_0)| \leq \delta\}}\, dy_{k, \ell}. 
\end{equation}
We lower bound the first and second expectation in the integral of \eqref{eq:twoexp}, assuming that $|y_{k, \ell} - x_k| \leq \delta$. 
For the second expectation, apply Lemma \ref{lem:FKGprod}, we get $\mathbb{E}[\prod_{k = 1}^n \prod_{\ell = 1}^{m_k}  Z_T (t-\fmt, y_{k, \ell}; t,  x_{k})] \geq \prod_{k = 1}^n \mathbb{E}[\prod_{\ell = 1}^{m_k}  Z_T (t-\fmt, y_{k, \ell}; t,  x_{k})]$. Apply Proposition \ref{prop:continuity} to the preceding right hand side (note that $|y_{k, \ell} - \xib_k (s_0)| \leq \delta$) together with \cite[Lemma 4.1]{corwin2020kpz}, we get 
\begin{equation}\label{eq:lwbd1}
\mathbb{E}\Big[\prod_{k = 1}^n \prod_{\ell = 1}^{m_k}  Z_T (t-\fmt, y_{k, \ell}; t,  x_{k})\Big] 
\gtrsim \exp\bigg(T\Big(\sum_{k = 1}^n \frac{(m_k^3 - m_k) \fmt}{24} - \frac{m_k (x_k - \xib_k (s_0))^2}{2s_0} + f_1(\delta)\Big) \bigg),
\end{equation}
where $\lim_{\delta \to 
0} f_1 (\delta) = 0$.

We examine the point masses that have merged at time $s = s_0$ and let $\vec{x}' = (x'_1, \ldots, x'_{n'})$ and $\vec{m}' = (m'_1, \ldots, m'_{n'})$ denote the locations and weights of them. 
Since a merging happens at time $s = s_0$, we know that $n' < n$.  
For the first expectation in \eqref{eq:twoexp}, use Proposition \ref{prop:continuity}, we get $\mathbb{E}[\prod_{k = 1}^n \prod_{\ell = 1}^{m_k} Z_T (t-\fmt, y_{k, \ell})] \geq \mathbb{E}[\prod_{k = 1}^n Z_T (t-s_0, \xib_k(s_0))^{m_k}] e^{T f_2 (\delta)}$. Rewrite $\prod_{k = 1}^n Z_T (t-s_0, \xib_k(s_0))^{m_k}$ as  $\prod_{k = 1}^{n'} Z_T (t-s_0, x'_{k})^{m'_{k}}$ and apply the induction hypothesis, we get a lower bound 
\begin{equation}\label{eq:lwbd2}
\mathbb{E}\Big[\prod_{k = 1}^n \prod_{\ell = 1}^{m_k} Z_T (t-\fmt, y_{k, \ell})\Big] \geq \mathbb{E}\Big[\prod_{k = 1}^{n'} Z_T (t-s_0, x_k')^{m'_k}\Big] e^{T f_2 (\delta)} \gtrsim \exp\big(T \gamma_3 (t - s_0, \vec{x}', \vec{m}') + T f_2 (\delta)\big),
\end{equation}
where $\lim_{\delta \to 0} f_2 (\delta) = 0$. Apply the lower bounds \eqref{eq:lwbd1}-\eqref{eq:lwbd2} to the right hand side of \eqref{eq:twoexp} and then let $\delta \to 0$, we conclude that 
\begin{equation*}
\mathbb{E}\Big[\prod_{i = 1}^n Z_T(t, x_i)^{m_i}\Big] \gtrsim \exp\bigg(T\Big(\sum_{k = 1}^n \frac{(m_k^3 - m_k) \fmt}{24} - \frac{m_k (x_k - \xib_k (s_0))^2}{2\fmt} + \gamma_3 (t-\fmt, \vec{x}', \vec{m}')\Big)\bigg) = e^{T\gamma_3 (t, \vec{x}, \vec{m})}. \qedhere
\end{equation*}
The last equality is due to Lemma \ref{lem:computation}.
\end{proof}	


\section{A continuity result for the moments}\label{sec:continuity}

The main result in this section is  Proposition \ref{prop:continuity}, which proves a continuity result for the moments of the SHE. We prove the following lemma as a preparation. 
\begin{lem}\label{lem:jointmomentbd}
Recall that $\htk(t, x) := \frac{1}{\sqrt{2\pi t}} e^{-\frac{x^2}{2t}}$. For $w_i, v_i \in \mathbb{R}$, $i = 1, \ldots, n$ and $T(t-r) \geq 2\pi$, we have 
\begin{align}\label{eq:jointmomentbd1}
&\mathbb{E}\Big[\prod_{i = 1}^n Z_T (r, v_i; t, w_i)\Big] \leq 
69 n! \exp\Big(\frac{n^3 T(t-r)}{24}\Big)
\prod_{i = 1}^n \htk (T(t-r), T(w_i - v_i)).\\
\label{eq:jointmomentbd2}
&\mathbb{E}\Big[\prod_{i = 1}^n Z_T (r, v_i; t, w_i)\Big] \geq \prod_{i = 1}^n \htk (T(t-r), T(w_i - v_i)).
\end{align}
\end{lem}
\begin{proof}
To prove \eqref{eq:jointmomentbd1}, we apply H\"{o}lder's inequality and get
\begin{align*}
\mathbb{E}\Big[\prod_{i = 1}^n Z_T (r, v_i; t, w_i)\Big] \leq \prod_{i = 1}^n \big(\mathbb{E}
[Z_T (r, v_i; t, w_i)^n]
\big)^{1/n} \leq  69 n! \exp\Big(\frac{n^3 T (t-r)}{24}\Big) \prod_{i = 1}^n \htk (T(t-r), T(w_i - v_i)),
\end{align*}
where the last equality follows from \cite[Lemma 4.1]{corwin2020kpz}.

To prove \eqref{eq:jointmomentbd2}, by Lemma \ref{lem:FKGprod} and $\mathbb{E}[Z_T (r, v_i; t, w_i)] = \htk(T(t-r), T(w_i -v_i))$, we have
\begin{equation*}
\mathbb{E}\Big[\prod_{i = 1}^n Z_T (r, v_i; t, w_i)\Big] \geq \prod_{i = 1}^n \mathbb{E}[Z_T (r, v_i; t, w_i)] = \prod_{i = 1}^n \htk (T(t-r), T(w_i - v_i)).
\end{equation*}
This concludes the lemma. 
\end{proof}

Let $\vec{w} = (w_1, \ldots, w_n)$. Define $\|\vec{w}\|_\infty = \max_{i = 1, \ldots, n} |w_i|$ and do it similarly for $ \vec{w}', \vec{v}, \vec{v}'$. 
The following proposition is the main result of this section. 
\begin{prop}\label{prop:continuity}
For arbitrary fixed $R, \e > 0$, $r < t$ and $n \in \mathbb{Z}_{\geq 1}$, there exists $\delta > 0$ such that  
\begin{equation*}
-\e \leq \liminf_{T \to \infty} T^{-1} \log \frac{\mathbb{E}[\prod_{i = 1}^n Z_T(r, v_i; t, w_i)]}{\mathbb{E}[\prod_{i = 1}^n Z_T (r, v_i'; t, w_i')]} \leq \limsup_{T \to \infty} T^{-1} \log \frac{\mathbb{E}[\prod_{i = 1}^n Z_T (r, v_i; t, w_i)]}{\mathbb{E}[\prod_{i = 1}^n Z_T (r, v_i'; t, w_i')]} \leq \e 
\end{equation*}
holds uniformly for $\vec{w}, \vec{w}', \vec{v}, \vec{v}'$ satisfying $\|\vec{w} - \vec{w}'\|_\infty, \|\vec{v} - \vec{v}'\|_\infty \leq \delta$ and $ \|\vec{w}\|_\infty, \|\vec{w}'\|_\infty, \|\vec{v}\|_\infty, \|\vec{v}'\|_\infty \leq R$.
\end{prop}
\begin{proof}
By the Feynman-Kac formula \eqref{eq:fmpolymer}, the stochastic process $\{Z_T (r, v; t, w)\}_{(w, v) \in \mathbb{R}}$ has the same probability distribution as $\{Z_T (r, w; t, v)\}_{(w, v) \in \mathbb{R}}$. Hence, we can assume $\vec{w}' = \vec{w}$ in the proposition. It suffices to prove that
\begin{equation*}
\limsup_{T \to \infty} T^{-1} \log \frac{\mathbb{E}[\prod_{i = 1}^n Z_T (r, v_i; t, w_i)]}{\mathbb{E}[\prod_{i = 1}^n Z_T (r, v_i'; t, w_i)]} \leq \e 
\end{equation*}
holds uniformly for $\vec{w}, \vec{v}$ and $\vec{v}'$. The proof of the lower bound for $\liminf$ can be obtained by swapping $\vec{v}, \vec{v}'$. 

Write the constant $C = C(n, R, r, t)$ to simplify the notation. It is enough to prove that
\begin{equation}\label{eq:enough}
\mathbb{E}\Big[\prod_{i = 1}^n Z_T (r, v_i; t, w_i)\Big] \leq C e^{\frac{1}{2}T \e} \mathbb{E}\Big[\prod_{i = 1}^n Z_T (r, v'_i; t, w_i)\Big] + \frac{1}{2} \mathbb{E}\Big[\prod_{i = 1}^n Z_T (r, v_i; t, w_i)\Big].
\end{equation}
We fix $t_0 \in (r, t)$ which will be specified later.  
Apply the semigroup identity
and then Fubini's theorem, we have
\begin{equation}
\label{eq:twoE}
\mathbb{E}\Big[\prod_{i = 1}^n Z_T (r, v_i; t, w_i)\Big] =  \int_{\mathbb{R}^{n}} T^{n} \mathbb{E}\Big[\prod_{i = 1}^n Z_T (r, v_i; \itmt,  y_i)\Big] \mathbb{E}\Big[\prod_{i = 1}^n Z_T (\itmt,  y_i;t, w_i)\Big] \prod_{i = 1}^n dy_i.
\end{equation}
Fix a constant $K$ that will be specified later. Let $A := \{\vec{y} = (y_1, \ldots, y_n) \in \mathbb{R}^n: \|\vec{y}\|_\infty \leq K\}$. Write the above integral $
\int_{\mathbb{R}^{n}} \ldots$ as $\msE_1 + \msE_2$ where $\msE_1 = \int_A \ldots$ and $\msE_2 = \int_{A^c} \ldots$. To prove \eqref{eq:enough}, we need to show that 
\begin{equation}\label{eq:desiredupbd}
\msE_1 \leq C e^{\frac{1}{2}T \e} \mathbb{E}\Big[\prod_{i = 1}^n Z_T (r, v'_i; t, w_i)\Big], \quad \msE_2 \leq \frac{1}{2} \mathbb{E}\Big[\prod_{i = 1}^n Z_T (r, v_i; t, w_i)\Big].
\end{equation} 
We first prove the upper bound for $\msE_2$. 
Apply 
\eqref{eq:jointmomentbd1} to upper bound the first and second expectation in the integrand of $\msE_2$, we get
\begin{equation}\label{eq:E2}
\msE_2 \leq 69n!  e^{\frac{n^3 T (t-r)}{24} }  \int_{A^c} T^{n} \prod_{i = 1}^n \htk (T(\itmt-r), T(y_i - v_i))  \htk (T(t-\itmt), T(w_i - y_i))  
dy_i.
\end{equation}
Take a large enough $K = (1 + t-r)(10R + 69n!)$, it is straightforward to see that for all $\|\vec{w}\|_\infty, \|\vec{v}\|_\infty \leq R$ and $t_0 \in (t, r)$, we have,
\begin{equation}\label{eq:temp}
69n!  e^{\frac{n^3 T (t-r)}{24} } \int_{A^c} T^n \frac{\prod_{i = 1}^n \htk(T(\itmt-r), T(y_i - v_i))  \htk(T(t-\itmt), T(w_i - y_i)) dy_i}{\prod_{i = 1}^n \htk (T(t-r), T(v_i - w_i))} \leq  \frac{1}{2}. 
\end{equation}
The inequality \eqref{eq:temp} can be proved by interpreting the fraction above as the probability density function of $n$ independent Brownian bridges starting from $T\vec{v}$ at time $0$ and ending at $T\vec{w}$ at time $T(t-r)$. The integral is the probability of the event that at least one of the Brownian bridges go beyond $[-TK, TK]$ at time $T(\itmt-r)$. 

Continue our proof and multiply both sides of \eqref{eq:temp} by $\prod_{i = 1}^n \htk (T(t-r), T(v_i - w_i))$, apply the resulting inequality to upper bound the right hand side of \eqref{eq:E2} and finally apply \eqref{eq:jointmomentbd2}, we get $\msE_2 \leq \frac{1}{2} \mathbb{E}[\prod_{i = 1}^n Z_T (r, v_i; t, w_i)]$. 

We proceed to upper bound $\msE_1$. Recall that $\msE_1 = \int_A \ldots$ where $\ldots$ is given by the integrand on the right hand side of \eqref{eq:twoE}. Apply \eqref{eq:jointmomentbd1} to upper bound the first expectation in the integrand, we have
\begin{equation}\label{eq:E1}
\msE_1 \leq C  e^{\frac{1}{24} n^3 T (\itmt-r)}  \int_{A} T^{n} \mathbb{E}\Big[\prod_{i = 1}^n Z_T (\itmt, y_i; t,  w_i)\Big] \prod_{i = 1}^n  \htk (T(\itmt-r), T(y_i - v_i)) dy_i.
\end{equation}
Define $\para(t, x) := -\frac{x^2}{2t}$. It is straightforward to check that for $\|\vec{y}\|_{\infty} \leq K$ and $\|\vec{v}\|_\infty$, $\|\vec{v}'\|_{\infty} \leq M$, we have
\begin{equation*}
\Big|\sum_{i = 1}^n (\para(T(\itmt-r), T(y_i - v_i)) - \sum_{i = 1}^n \para(T(\itmt-r), T(y_i - v'_i))\Big| \leq \frac{C T \|\vec{v} - \vec{v}'\|_\infty}{\itmt-r}. 
\end{equation*}
Take the exponential, the above inequality implies that 
\begin{equation}\label{eq:ratio}
\prod_{i = 1}^n \htk(T(\itmt-r), T(y_i - v_i)) \leq \exp\Big(\frac{C T \|\vec{v} - \vec{v}'\|_\infty}{\itmt-r}\Big) \prod_{i = 1}^n \htk(T(\itmt-r), T(y_i - v'_i)).
\end{equation}
Apply \eqref{eq:ratio} to upper bound the right hand side of \eqref{eq:E1}, and then release the domain of integral to $\mathbb{R}^n$, we have  
\begin{equation*}
\msE_1 \leq C e^{\frac{1}{24} n^3 T (\itmt-r)} \exp\Big(\frac{C T \|\vec{v} - \vec{v}'\|_\infty}{\itmt-r}\Big)  \int_{\mathbb{R}^n} T^n \mathbb{E}\Big[\prod_{i = 1}^n Z_T(\itmt, y_i; t, w_i)\Big] \prod_{i = 1}^n  \htk(T(\itmt-r), T(y_i - v'_i)) dy_i.
\end{equation*}
Apply \eqref{eq:jointmomentbd2} to upper bound $\prod_{i = 1}^n \htk (T(\itmt-r), T(y_i - v'_i))$ by $\mathbb{E}[\prod_{i = 1}^n Z_T (r, v'_i; \itmt, y_i)]$ and then use the Fubini's theorem together with the semigroup identity, we get 
\begin{equation*}
\msE_1 \leq C e^{\frac{1}{24} n^3 T (\itmt-r)} \exp\Big(\frac{C T \|\vec{v} - \vec{v}'\|_\infty}{\itmt-r}\Big) \mathbb{E}\Big[\prod_{i = 1}^n Z_T (r, v_i'; t, w_i)\Big]. 
\end{equation*}
Take $\delta = \min(\frac{\e^2}{16 C^2}, \frac{36\e^2}{n^6})$, $\itmt = r + \delta^{\frac{1}{2}}$. Use $\|\vec{v} - \vec{v}'\|_\infty \leq \delta$, we obtain the  upper bound \eqref{eq:desiredupbd} and thus \eqref{eq:enough}.
\end{proof}
\section{Proof of Theorem \ref{thm:general} and its corollaries}
\label{sec:ordering}
\noindent In this section, we will prove the following lemma and conclude the proof of Theorem \ref{thm:general}.
\begin{lem}\label{lem:gamma}
We have $\gamma_1 \leq \gamma_2 \leq \gamma_3$.
\begin{proof}
We first prove $\gamma_1 \leq 
\gamma_2$ by finding $a_1, \ldots, a_{\totnum}$ which satisfy $a_{i} - a_{i+1} \geq 1$ for $i = 1, \ldots, \totnum - 1$ and $\sum_{i = 1}^{\totnum} \frac{t}{2} a_i^2 + u_i a_i = \gamma_2 (t, \vec{x}, \vec{m})$. 
Let $(b_1, \dots, b_n)$ be a minimizer of $\gamma_2$. For $j = 1, \ldots, n$, set $S_j := \sum_{i = 1}^j m_i$ and
\begin{equation}\label{eq:ab}
a_k := b_j + \frac{m_j + 1}{2} - k + \sum_{i = 1}^{j-1} m_i, \qquad \text{ if } S_{j-1} < k \leq S_j.
\end{equation} 
Note that we have $a_{k} - a_{k+1} = 1$ unless $k = S_j$ for some $j \in \{1, \ldots, n-1\}$. Since $b_j - b_{j+1} \geq \frac{m_j + m_{j+1}}{2}$ for $j = 1, \dots, n-1$, we know that $a_{k} - a_{k+1} \geq 1$ for $k \in \{S_1, \ldots, S_{n-1}\}$. This implies that $a_{k} - a_{k+1} \geq 1$ for all $k = 1, \ldots, \totnum-1$.
Recall that $u_k = x_j$ if $S_{j-1} < k \leq S_j$, one can directly check that 
\begin{equation}\label{eq:gammarelation1}
\sum_{j = 1}^\totnum \frac{t}{2} a_j^2 + \sum_{j = 1}^\totnum u_j a_j = \sum_{i = 1}^n  \sum_{k = S_{i - 1}+1}^{S_i} \frac{t}{2} a_k^2 + x_i a_k =  \sum_{i = 1}^n \frac{m_i t}{2} (b_i +\frac{x_i}{t})^2 + \frac{(m_i^3 - m_i) t}{24} - \frac{m_i x_i^2}{2t} = \gamma_2 (t, \vec{x}, \vec{m}).
\end{equation}
This implies that $\gamma_1 \leq \gamma_2$. 

We proceed to show $\gamma_2 \leq \gamma_3$ by finding $b_1, \ldots, b_n$ which satisfy $b_i  - b_{i+1} \geq \frac{m_i + m_{i+1}}{2}$ for $i = 1, \ldots, n-1$ and $\sum_{i = 1}^n \frac{m_i t}{2} (b_i +\frac{x_i}{t})^2 + \frac{(m_i^3 - m_i) t}{24} - \frac{m_i x_i^2}{2t} = \gamma_3 (t, \vec{x}, \vec{m})$. Recall the notation for the inertia clusters from Section \ref{sec:inertia}.
Let $n_i := |\cls_i|$ for $i = 1, \dots, \cnf$ and $\clsum_k := \sum_{i = 1}^k n_i$. Then, we have $\cls_k = \{\clsum_{k-1} + 1, \dots, \clsum_k\}$. Let 
\begin{equation*}
\cm_k (s) := \frac{\sum_{i = \clsum_{k - 1} + 1}^{\clsum_k} m_i  \zetab_i (s)}{\totnum_{\cls_{k}}}.
\end{equation*}
Note that $\cm_k (s)$ is the location of the center of mass for $\{\zetab_i (s)\}_{i \in \cls_k}$. An important observation is that $\cm_k$ travels with a constant speed $\cmspd_k := \frac{1}{2}(\sum_{i = \clsum_k +1}^{n} m_i - \sum_{i = 1}^{\clsum_{k-1}} m_i)$. Since $\zetab_i (0) = x_i$, we have $\cm_k (0) = \frac{\sum_{i = \clsum_{k - 1} + 1}^{\clsum_k} m_i  x_i}{\totnum_{\cls_k}}$. 
Moreover, we have  $\cm_k (t) = \zetab_{\cls_k} (t)$, which implies that $\cm_k (t) < \cm_{k+1} (t)$ for all $k = 1, \dots, \cnf - 1$. Use this together with $\cm_j (t) = \varphi_j t + \cm_j (0)$, we have
$\cm_{k+1} (0) - \cm_k (0) > -(\varphi_{k+1} - \varphi_k) t$, which is equivalent to 
\begin{equation}\label{eq:separate}
\frac{\sum_{i = \clsum_{k} + 1}^{\clsum_{k+1}} m_i  x_i}{\totnum_{\cls_{k+1}}} - \frac{\sum_{i = \clsum_{k - 1} + 1}^{\clsum_k} m_i  x_i}{\totnum_{\cls_{k}}} > \frac{(\sum_{i = \clsum_{k-1} + 1}^{\clsum_{k+1}} m_i)t}{2}.
\end{equation}
We set  
\begin{equation}\label{eq:bcls}
b_i := -\frac{\sum_{j = \clsum_{k-1} + 1}^{\clsum_k} m_j x_j}{\totnum_{\cls_{k}} t} + \frac{\sum_{j = i+1}^{\clsum_k} m_j - \sum_{j = N_{k-1}+1}^{i-1} m_j}{2}, \qquad \text{ for } \clsum_{k - 1} < i \leq \clsum_{k}.
\end{equation} 
By \eqref{eq:separate}, one can verify that $b_{i} - b_{i+1} \geq \frac{m_i + m_{i+1}}{2}$ for $i = 1, \ldots, n-1$. Recall that 
\begin{equation*}
\gamma_3 (t, \vec{x}_{\cls_j}, \vec{m}_{\cls_j})  = \frac{(\mset_{\cls_j}^3 - \mset_{\cls_j})t}{24} - \sum_{k, \ell \in \cls_j, k < \ell} \frac{m_k m_\ell}{2} |x_k - x_\ell| - \frac{(\sum_{k \in \cls_j} m_k x_k)^2}{2 t \mset_{\cls_j}}.
\end{equation*}
A straightforward (although tedious) computation implies that
\begin{equation}\label{eq:gammarelation2}
\sum_{i = \clsum_{j-1}+1}^{\clsum_j} \frac{m_i t}{2} (b_i +\frac{x_i}{t})^2 + \frac{(m_i^3 - m_i) t}{24} - \frac{m_i x_i^2}{2t} = \gamma_3 (t, \vec{x}_{\cls_j}, \vec{m}_{\cls_j}). 
\end{equation}
Summing both sides above over $j = 1, \ldots, \cnf$, we have 
\begin{equation*}
\sum_{i = 1}^n \frac{m_i t}{2} (b_i +\frac{x_i}{t})^2 + \frac{(m_i^3 - m_i) t}{24} - \frac{m_i x_i^2}{2t} = \gamma_3 (t, \vec{x}, \vec{m}).
\end{equation*}
This shows that $\gamma_2 \leq \gamma_3$. 
\end{proof}
\end{lem}
\begin{proof}[Proof of Theorem \ref{thm:general}]
Apply Propositions \ref{prop:upbd} and \ref{prop:lwbd}, we know that $\gamma_1 \geq \gamma_3$. Use this together Lemma \ref{lem:gamma}, we have $\gamma_1 = \gamma_2  = \gamma_3$. Use this together with Propositions \ref{prop:upbd} and \ref{prop:lwbd}, we conclude Theorem \ref{thm:general}.
\end{proof}
\begin{proof}[Proof of Corollary \ref{cor:explicit}]
Use the expression of $\gamma_2$, the result is straightforward for $n = 1$. When $n = 2$, we have 
	\begin{equation*}
	\gamma_2 (t, \vec{x}, \vec{m}) = \gamma_2 (\ft, \vec{x}, \vec{m}) := \inf\Big\{\sum_{i = 1}^2 \frac{m_i \ft}{2} (b_i +\frac{x_i}{\ft})^2 + \frac{(m_i^3 - m_i) \ft}{24} - \frac{m_i x_i^2}{2\ft}: b_1 - b_{2} \geq \frac{m_1 + m_{2}}{2}\Big\}.
	\end{equation*}
	When $\frac{x_2 - x_1}{t} \geq \frac{m_1 + m_2}{2}$, the target function is minimized at $b_1 = -\frac{x_1}{t}$ and $b_2 = -\frac{x_2}{t}$. When $0 < \frac{x_2 - x_1}{t} \leq \frac{m_1 + m_2}{2}$, the target function in the infimum is minimized at the boundary $b_1 - b_2 = \frac{m_1 + m_2}{2}$ with $b_1 = -\frac{m_1 x_1 + m_2 x_2}{(m_1 + m_2) t} + \frac{m_2}{2}$. Insert the minimizers into the target function, we obtain the desired result. 
\end{proof}
\begin{proof}[Proof of Corollary \ref{cor:minimizer}]
Since the target function that we want to minimize in $\gamma_1$ is continuous and goes to infinity when we send $\sum_{i = 1}^\totnum |a_i|$ to infinity, we know that the infimum in $\gamma_1$ has a minimizer. Moreover, the domain $\{(a_1, \ldots, a_{\totnum}): a_{i} - a_{i+1} \geq 1, \forall\, i \in \{1, \ldots, \totnum - 1\}\}$ is convex and the target function is strictly convex, thus the infimum in $\gamma_1$ has a unique minimizer $(\mra_1, \ldots, \mra_{\totnum})$. By a similar argument, we know that $\gamma_2$ has a unique minimizer $(\mrb_1, \ldots, \mrb_{n})$. Use \eqref{eq:ab} - \eqref{eq:gammarelation1} and $\gamma_1 = \gamma_2$, we have  
\begin{equation}\label{eq:a}
\mra_k = \mrb_j + \frac{m_j + 1}{2} - k + \sum_{i = 1}^{j-1} m_i, \qquad \text{ if } S_{j-1} < k \leq S_j.
\end{equation}
Moreover, use \eqref{eq:bcls} -  \eqref{eq:gammarelation2} and $\gamma_2 = \gamma_3$, we know that 
\begin{equation*}
\mrb_i := -\frac{\sum_{j = \clsum_{k-1} + 1}^{\clsum_k} m_j x_j}{\totnum_{\cls_{k}} t} + \frac{\sum_{j = i+1}^{\clsum_k} m_j - \sum_{j = N_{k-1}+1}^{i-1} m_j}{2}, \qquad \text{ for } \clsum_{k - 1} < i \leq \clsum_{k}.
\end{equation*}
By \eqref{eq:separate}, we have $\mrb_i - \mrb_{i+1} = \frac{m_i + m_{i+1}}{2}$ if $i, i+1$ belong to the same $\cls_j$ and $\mrb_i - \mrb_{i+1} > \frac{m_i + m_{i+1}}{2}$ if not. 
Use this together with \eqref{eq:a}, we conclude that $\mra_{i} - \mra_{i+1} = 1$ if and only if $f(i)$ and $f(i+1)$ belong to the same $\cls_j$.
\end{proof}
\appendix
\section{A correlation inequality}

In this section, we prove the following inequality. 
\begin{lem}\label{lem:FKGprod}
	For any $T > 0$ and $t > r$, integers $1 \leq k \leq n$, and real numbers $w_1, \ldots, w_n$, $v_1, \ldots, v_n$, we have 
	\begin{equation*}
	\mathbb{E}\Big[\prod_{i = 1}^n Z_T (r, v_i; t, w_i)\Big] \geq \mathbb{E}\Big[\prod_{i = 1}^k Z_T (r, v_i; t, w_i)\Big] \mathbb{E}\Big[\prod_{i = k+1}^n Z_T (r, v_i; t, w_i)\Big].
	\end{equation*}
\end{lem}
\begin{proof}
Without loss of generality, we can take $T = 1$ and write $Z_1$ as $Z$.
	We claim that for any positive real numbers $s_1, \ldots, s_n$, we have 
	\begin{equation}\label{eq:FKG2}
	\mathbb{P}\Big(\bigcap_{i = 1}^n \{Z(r, v_i; t, w_i) \geq s_i\}\Big) \geq \mathbb{P}\Big(\bigcap_{i = 1}^k \{Z(r, v_i; t, w_i) \geq s_i\}\Big) \mathbb{P}\Big(\bigcap_{i = k+1}^n \{Z(r, v_i; t, w_i) \geq s_i\}\Big).
	\end{equation}
	The proof of the claim follows the idea of \cite[Proposition 1]{corwin2013crossover} and uses the FKG-Harris inequality (see \cite[Proposition A.1]{comets2017directed}) at the level of the discrete polymer model. By \cite[Theorem 2.7]{alberts2014intermediate}, we 
	can approximate the four-parameter process $Z(r, y; t, x)$ in terms of the partition function of discrete polymer models, which is denoted as $\Zdc_\e (r, y; t, x)$. More precisely, at the process level, $\Zdc_\e (r, y; t, x)$ converges in distribution to $Z(r, y; t, x)$ as $\e \to 0$. It is straightforward that $\Zdc_\e (r, y; t, x)$ is an increasing function of the i.i.d. random variables that we put on the lattice $\mathbb{Z}^2$ in the discrete polymer models. Hence, the events $A_{1, \e} = \cap_{i = 1}^k \{\Zdc_\e (r, v_i; t, w_i) \geq s_i\}$ and $A_{2, \e} = \cap_{i = k+1}^n \{\Zdc_\e (r, v_i; t, w_i) \geq s_i\}$ are increasing events.  
	Note that although the lattice $\mathbb{Z}^2$ is infinite, however, only a
	finite number of random variables on the lattice affects the value of $\Zdc_\e (r, v_i; t, w_i)$ for $i = 1, \ldots, n$. By the FKG-Harris inequality,  
$\mathbb{P}(A_{1, \e} \cap A_{2, \e}) \geq \mathbb{P}(A_{1, \e}) \mathbb{P}(A_{2, \e}).$ Send $\e \to 0$, we conclude \eqref{eq:FKG2}. 

	We proceed to conclude Lemma \ref{lem:FKGprod}. Apply Fubini's theorem, we have
	\begin{equation*}
	\mathbb{E}\Big[\prod_{i = 1}^n  Z(r, v_i; t, w_i)\Big] = \int_{\mathbb{R}_{> 0}^n} \mathbb{P}\Big(\bigcap_{i = 1}^n \{Z(r, v_i; t, w_i) \geq s_i\}\Big) ds_1 \ldots ds_n.
	\end{equation*}
	Apply  \eqref{eq:FKG2} to the right hand side and then use Fubini's theorem, we have
	\begin{align*}
	\mathbb{E}\Big[\prod_{i = 1}^n  Z(r, v_i; t, w_i)\Big] &\geq \int_{\mathbb{R}_{> 0}^k} \mathbb{P}\Big(\bigcap_{i = 1}^k \{Z(r, v_i; t, w_i) \geq s_i\}\Big) ds_1 \ldots ds_k \int_{\mathbb{R}_{> 0}^{n-k}} \mathbb{P}\Big(\bigcap_{i = k+1}^n \{Z(r, v_i; t, w_i) \geq s_i\}\Big) ds_{k+1} \ldots ds_n\\
	&= \mathbb{E}\Big[\prod_{i = 1}^k  Z(r, v_i; t, w_i)\Big] \mathbb{E}\Big[\prod_{i = k+1}^n  Z(r, v_i; t, w_i)\Big]. \qedhere
	\end{align*}
\end{proof}
\section{An identity}
\noindent In this section, we will prove the last equality in the proof of Proposition \ref{prop:lwbd}. Recall that for the optimal clusters starting from $(\vec{x}, \vec{m})$, we have assumed $\cnf = 1$ and set $s_0 = \inf\{s: |\{\xib_1 (s), \dots, \xib_n (s)\}| < n\}$. We let $(\vec{x}', \vec{m}')$ be the (different) locations and weights of the point masses at time $s = s_0$. Note that $(\vec{x}', \vec{m}')$ is obtained from $(\vec{\xib}(s_0), \vec{m})$ by examining the point masses that have merged.
\begin{lem}\label{lem:computation}
We have
\begin{equation}\label{eq:verify}
\sum_{k = 1}^n \frac{(m_k^3 - m_k) \fmt}{24} - \frac{m_k (x_k - \xib_k (s_0))^2}{2\fmt} + \gamma_3 (t-\fmt, \vec{x}', \vec{m}') = \gamma_3 (t, \vec{x}, \vec{m}).
\end{equation} 
\end{lem}
\begin{proof}
Let $k \in \{1, \ldots, n-1\}$ be the number such that $\frac{x_{k+1} - x_k}{(m_k + m_{k+1})/2} = \min_{j \in \{1, \ldots, n-1\}} \frac{x_{j+1} - x_j}{(m_j + m_{j+1})/2}$, the point masses starting from $x_k$ and $x_{k+1}$ will merge first at time $s_0 = \frac{x_{k+1} - x_k}{(m_k + m_{k+1})/2}$. It is not hard to check that $\xib_k (s_0) = x_k - (\phi_k + \spd) s_0$ with $$\phi_k = \frac{\sum_{j = k+1}^n m_j - \sum_{j = 1}^{k-1} m_j}{2} \text{ and } v = \frac{\sum_{j = 1}^n m_j x_j}{\totnum t}.$$ Since $\cnf = 1$, we have 
\begin{align*}
&\gamma_3 (t, \vec{x}, \vec{m}) = \frac{(\totnum^3 - \totnum)t}{24} - \sum_{1 \leq j < k \leq n} \frac{m_j m_k (x_k  -x_j)}{2} - \frac{(\sum_{j = 1}^n m_j x_j)^2}{2\totnum t},\\
&\gamma_3 (t-s_0, \vec{x}', \vec{m}') = \frac{(\totnum^3 - \totnum) (t-s_0)}{24} - \sum_{1 \leq j < k \leq n} \frac{m_j m_k (\xib_k(s_0) - \xib_j(s_0))}{2} - \frac{(\sum_{j = 1}^n m_j \xib_j (s_0))}{2\totnum (t-s_0)}.
\end{align*} 
Now we have an explicit expression for both sides of \eqref{eq:verify} in terms of $(\vec{x}, \vec{m}, t)$, a straightforward (though tedious) computation verifies \eqref{eq:verify}. 
\end{proof}

		
		\bibliographystyle{alpha}
		\bibliography{ref}
		
	\end{document}